\documentclass{amsart}
\usepackage{amsmath,amsthm,amsfonts, amssymb,amscd}
\usepackage[all]{xy}

\newtheorem{theorem}{Theorem}[section]
\newtheorem{definition}[theorem]{Definition}
\newtheorem{corollary}[theorem]{Corollary}
\newtheorem{proposition}[theorem]{Proposition}

\newcommand{\minusre}{\hspace{0.3em}\raisebox{0.3ex}{\sl \tiny /}\hspace{0.3em}}
\newcommand{\minusli}{\hspace{0.3em}\raisebox{0.3ex}{\sl \tiny $\setminus $}\hspace{0.3em}}
\newcommand{\lex}{\,\overrightarrow{\times}\,}

\newcommand{\Ker}{\mbox{\rm Ker}}

\newcommand{\Rad}{\mbox{\rm Rad}}

\newcommand{\Infinit}{\mbox{\rm Infinit}}

\newcommand{\RDP}{\mbox{\rm RDP}}

\begin{document}
\title[$\mathbb H$-perfect Pseudo MV-algebras]{$\mathbb H$-perfect Pseudo MV-algebras and Their Representations}
\author[Anatolij Dvure\v{c}enskij]{Anatolij Dvure\v censkij$^{1,2}$}
\date{}%
\maketitle
\begin{center}  \footnote{Keywords: Pseudo MV-algebra, $\ell$-group, unital $\ell$-group, strong unit, Riesz Decomposition Property, lexicographic product, $H$-perfect pseudo MV-algebra, strong $\mathbb H$-perfect pseudo MV-algebra, weak $\mathbb H$-perfect pseudo MV-algebra

 AMS classification: 03G12 81P15,  03B50

The paper has been supported by  Slovak Research and Development Agency under the contract APVV-0178-11, the grant VEGA No. 2/0059/12 SAV, and by
CZ.1.07/2.3.00/20.0051.
 }
Mathematical Institute,  Slovak Academy of Sciences,\\
\v Stef\'anikova 49, SK-814 73 Bratislava, Slovakia\\
$^2$ Depart. Algebra  Geom.,  Palack\'{y} University\\
17. listopadu 12, CZ-771 46 Olomouc, Czech Republic\\

E-mail: {\tt dvurecen@mat.savba.sk}
\end{center}

\begin{center}
\small{\it Dedicated  to Prof. Antonio Di Nola on the occasion of his $65^{th}$ birthday}
\end{center}

\begin{abstract}
We study $\mathbb H$-perfect pseudo MV-algebras, that is, algebras which can be split into a system of ordered slices indexed by the elements of an subgroup $\mathbb H$ of the group of the real numbers. We show when they can be represented as a lexicographic product of $\mathbb H$ with some $\ell$-group. In addition, we show also a categorical equivalence of this category with the category of $\ell$-groups.
\end{abstract}

\section{Introduction}

MV-algebras were introduced by Chang in \cite{Cha} in order to provide an algebraic counterparts of infinite-valued sentential calculus of \L ukasiewicz logic. Thanks to the celebrated Representation Theorem by Mundici \cite{Mun}, such algebras are always an interval in  Abelian $\ell$-groups with strong unit, see also e.g. \cite{CDM}. Recently, there appeared independently two non-commutative generalizations of MV-algebras, called pseudo MV-algebras by \cite{GeIo} and generalized MV-algebras by \cite{Rac}, which are both equivalent. The basic result on pseudo MV-algebras from \cite{151} says that every pseudo MV-algebra is an interval in a unital $\ell$-group with strong unit which is not necessarily Abelian.

A more general structure than MV-algebras is formed by effect algebras \cite{FoBe} which are partial algebras important for modeling quantum mechanical measurements. Such algebras are also sometimes an interval in Abelian partially ordered groups (po-groups) with strong unit. This is possible e.g. if the effect algebra has the Riesz Decomposition Property, \cite{Rav}. For more on effect algebras, see \cite{DvPu}. A noncommutative version of effect algebras, called pseudo MV-algebras, was presented in \cite{DvVe1, DvVe2}. Also under a stronger type of the Riesz Decomposition Property, such algebras are intervals in po-groups with strong unit which are not necessarily Abelian. It is important to note that every pseudo MV-algebra can be viewed also as a pseudo effect algebra satisfying RDP$_2$, see \cite{DvVe2}.

We recall that a {\it po-group} (= partially ordered group) is a
group $(G;+,0)$ (written additively) endowed with a partial order $\le$ such that if $a\le b,$ $a,b \in G,$ then $x+a+y \le x+b+y$ for all $x,y \in G.$  We denote by $G^+=\{g \in G: g \ge 0\}$ the {\it positive cone} of $G.$ If, in addition, $G$
is a lattice under $\le$, we call it an $\ell$-group (= lattice
ordered group).  An element $u\in G^+$ is said to be a {\it strong unit}
(= order unit) if given $g \in G,$ there is an integer $n\ge 1$ such
that $g \le nu,$ and the couple $(G,u)$ with a fixed strong unit $u$ is
said to be a {\it unital po-group} and a {\it unital $\ell$-group}, respectively. For more information on po-groups and $\ell$-groups and for unexplained notions, see \cite{Fuc, Gla}.

We say that an MV-algebra is perfect if every its element is either an infinitesimal or the negation of some infinitesimal. Therefore, they are mostly non Archimedean algebras. An important example of a perfect MV-algebra is the subalgebra of the Lindenbaum algebra  of the first order \L ukasiewicz logic generated by the class of formulas that are valid but non-provable, \cite{DDT}. Hence, perfect MV-algebras are directly connected with the very important phenomenon of incompleteness of the \L ukasiewicz first-order logic. Important results on perfect pseudo MV-algebras can be found in \cite{DiLe1} together with their equational characterization. This notion was extended also for effect algebras in \cite{177}.  Perfect pseudo MV-algebras were studied in \cite{Leu} and \cite{DDT}, where it was shown that such algebras are always of the form $\Gamma(\mathbb Z \lex G,(1,0)),$ where $G$ is an $\ell$-groups. This notion was generalized for the so-called $n$-perfect pseudo MV-algebras, \cite{Dv08}. Such algebras can be split into $n+1$ comparable slices, see e.g. \cite{DXY}. This notion was exhibited also for the case when a pseudo effect algebra can be split into a system of comparable slices indexed by the elements of a subgroup $\mathbb H$ of the group of real numbers $\mathbb R,$ see \cite{DvKo}. We note that the structure of perfect pseudo MV-algebras is very rich because there is uncountably many varieties of pseudo MV-algebras generated by the categories of perfect pseudo MV-algebras, see \cite{DDT}.

In the present paper, we study $\mathbb H$-perfect pseudo MV-algebras. We introduce so-called strong $\mathbb H$-perfect pseudo MV-algebras as algebras which can be represented as $\Gamma(\mathbb H \lex G, (1,0)),$ where $G$ is an $\ell$-group. We present also their categorical representation by the category of $\ell$-group. In addition, we introduce also weak $\mathbb H$-perfect pseudo MV-algebras as algebras which can be represented in the form $\Gamma(\mathbb H \lex G,(1,b)),$ where $b$ is a strictly positive element of an $\ell$-group $G.$

The paper is organized as follows.  Section 2 gathers elements of pseudo MV-algebras and pseudo effect algebras. Section 3 introduces $\mathbb H$-perfect pseudo MV-algebras. Section 4 deals with strong $\mathbb H$-perfect pseudo MV-algebras and it gives a representation theorem for such algebras. Section 5 shows a categorical equivalence of the category of strong $\mathbb H$-perfect pseudo MV-algebras with the category of $\ell$-groups.  Finally, Section 6 presents a representation of weak $\mathbb H$-perfect pseudo MV-algebras together with their categorical equivalence.

\section{Pseudo MV-algebras}

According to \cite{GeIo}, a {\it pseudo MV-algebras} ({\it PMV-algebra} for short) is an algebra $(M;
\oplus,^-,^\sim,0,1)$ of type $(2,1,1,$ $0,0)$ such that the
following axioms hold for all $x,y,z \in M$ with an additional
binary operation $\odot$ defined via $$ y \odot x =(x^- \oplus y^-)
^\sim $$
\begin{enumerate}
\item[{\rm (A1)}]  $x \oplus (y \oplus z) = (x \oplus y) \oplus z;$

\item[{\rm (A2)}] $x\oplus 0 = 0 \oplus x = x;$

\item[{\rm (A3)}] $x \oplus 1 = 1 \oplus x = 1;$

\item[{\rm (A4)}] $1^\sim = 0;$ $1^- = 0;$

\item[{\rm (A5)}] $(x^- \oplus y^-)^\sim = (x^\sim \oplus y^\sim)^-;$

\item[{\rm (A6)}] $x \oplus (x^\sim \odot y) = y \oplus (y^\sim
\odot x) = (x \odot y^-) \oplus y = (y \odot x^-) \oplus
x;$\footnote{$\odot$ has a higher priority than $\oplus$.}

\item[{\rm (A7)}] $x \odot (x^- \oplus y) = (x \oplus y^\sim)
\odot y;$

\item[{\rm (A8)}] $(x^-)^\sim= x.$
\end{enumerate}

For example, if $u$ is a strong unit of a (not necessarily Abelian)
$\ell$-group $G$,
$$\Gamma(G,u) := [0,u]
$$
and
\begin{eqnarray*}
x \oplus y &:=&
(x+y) \wedge u,\\
x^- &:=& u - x,\\
x^\sim &:=& -x +u,\\
x\odot y&:= &(x-u+y)\vee 0,
\end{eqnarray*}
then $(\Gamma(G,u);\oplus, ^-,^\sim,0,u)$ is a PMV-algebra
\cite{GeIo}.

(A6) defines the join $x\vee y$ and (A7) does the meet $x\wedge y.$ In addition, $M$ with respect to $\vee$ and $\wedge$ is a distributive lattice, \cite{GeIo}.

Let $(M; \oplus,^-,^\sim,0,1)$ be a PMV-algebra. Define a partial
binary operation $+$ on $M$ via: $x + y $ is defined iff $x \le
y^-$, and in this case
$$
 x+ y:= x \oplus y.\eqno(2.1)
$$

A PMV-algebra is an {\it MV-algebra} if $a\oplus b= b\oplus a$ for all $a,b \in M.$ We denote by $\mathcal{PMV}$ and $\mathcal{MV}$ the variety of pseudo MV-algebras and MV-algebras, respectively.

A PMV-algebra is said to be {\it symmetric} if $a^-=a^\sim$ for any $a \in M.$ We recall that a symmetric PMV-algebra is not necessarily an MV-algebra, see e.g. the PMV-algebra $M=\Gamma(\mathbb H \lex G,(1,g_0)),$ where $g_0 >0 $ is not from the {\it commutative center} $C(G):=\{x \in G: x+y=y+x,\ \forall\ y\in G\}.$ The class of all symmetric
PMV-algebras forms a variety, $\mathcal{SYM}$, which contains as a
proper subvariety the variety of all MV-algebras.

If $A$ is a non-void subset of a PMV-algebra $M,$ we set
$A^-:=\{a^-: a \in A\},$ $A^\sim:= \{a^\sim : a \in A\}$ and if $B$ is another non-void subset of $M,$ we write $A \leqslant B$ if $a\le b$ for all $a \in A$ and all $b \in B.$

An {\it ideal} of a PMV-algebra $M$ is any non-empty subset $I$ of $M$ such that (i) $a\le b \in I$ implies $a \in I,$ and (ii) if $a,b \in I,$ then $a\oplus b \in I.$
An ideal $I\ne M$ is said to be {\it maximal} if it is not a proper subset of another ideal $J \ne M;$ we denote by $\mathcal M(M)$ the set of maximal ideals of $M.$

According to \cite{DvVe1, DvVe2}, a partial algebraic structure
$(E;+,0,1),$  where $+$ is a partial binary operation and 0 and 1 are constants, is called a {\it pseudo effect algebra} ({\it PEA} for short)  if, for all $a,b,c \in E,$ the following hold.
\begin{itemize}
\item[{\rm (PE1)}] $ a+ b$ and $(a+ b)+ c $ exist if and only if $b+ c$ and $a+( b+ c) $ exist, and in this case,
$(a+ b)+ c =a +( b+ c).$

\item[{\rm (PE2)}] There are exactly one  $d\in E $ and exactly one $e\in E$ such
that $a+ d=e + a=1.$

\item[{\rm (PE3)}] If $ a+ b$ exists, there are elements $d, e\in E$ such that
$a+ b=d+ a=b+ e.$

\item[{\rm (PE4)}] If $ a+ 1$ or $ 1+ a$ exists,  then $a=0.$
\end{itemize}

If we define $a \le b$ if and only if there exists an element $c\in
E$ such that $a+c =b,$ then $\le$ is a partial ordering on $E$ such
that $0 \le a \le 1$ for any $a \in E.$ It is possible to show that
$a \le b$ if and only if $b = a+c = d+a$ for some $c,d \in E$. We
write $c = a \minusre b$ and $d = b \minusli a.$ Then

$$ (b \minusli a) + a = a + (a \minusre b) = b,
$$
and we write $a^- = 1 \minusli a$ and $a^\sim = a\minusre 1$ for any
$a \in E.$

If $(G,u)$ is a unital po-group, the set $\Gamma(G,u):=\{g\in G: 0\le g \le u\}$ endowed with the restriction of the group addition $+$ to $\Gamma(G,u)$ and $0,u$ is a pseudo effect algebra.

Let $x\in M$  and an integer $n\ge 0$  be given. We define
$$
0\odot x :=0,\quad  1\odot x:=x, \quad (n+1)\odot x := (n\odot x)\oplus x,
$$
$$
x^0:= 1, \quad x^1 := x,\quad x^{n+1} := x^n \odot x,
$$
$$
0x := 0,\quad 1 x :=x, \quad (n+1)x:= (nx)+ x,
$$
if $nx$ and $(nx)+x$ are defined in $M.$ An element $x$ is said to be an {\it infinitesimal} if $mx$ exists in $M$ for any integer $m \ge 1.$ We denote by $\Infinit(M)$ the set of all infinitesimals of $M.$

A non-empty subset $I$ of a PEA $E$ is said to be an {\it ideal} if (i) $a,b\in I,$ $a+b\in E,$ then $a+b \in I,$ and (ii) if $a\le b \in I,$ then $a \in I.$

We introduce the following types of the Riesz Decomposition properties of po-groups:

\begin{enumerate}

\item[(i)]
\RDP\  if, for all $a_1,a_2,b_1,b_2 \in G^+$ such that $a_1 + a_2 = b_1+b_2,$ there are four elements $c_{11},c_{12},c_{21},c_{22}\in G^+$ such that $a_1 = c_{11}+c_{12},$ $a_2= c_{21}+c_{22},$ $b_1= c_{11} + c_{21}$ and $b_2= c_{12}+c_{22};$

\item[(ii)]
\RDP$_1$  if it satisfies RDP and, for the elements $c_{12}$ and $c_{21},$ we have $0\le x\le c_{12}$ and $0\le y \le c_{21}$ imply  $x+y=y+x;$

\item[(iii)]
\RDP$_2$  if it satisfies RDP and, for the elements $c_{12}$ and $c_{21},$ we have $c_{12}\wedge c_{21}=0.$

\end{enumerate}

If $E$ is a pseudo effect algebra, we say that $E$ satisfies RDP (or RDP$_1$ or RDP$_2$) if in the later definition we change $G^+$ to $E.$ Then RDP$_2$ implies RDP$_1,$ and RDP$_1$ implies RDP; but the converse is not true, in general. A po-group $G$ satisfies \RDP$_2$ iff $G$ is an $\ell$-group, \cite[Prop 4.2(ii)]{DvVe1}.

The basic results on PMV-algebras and PEAs are the following representation theorems \cite{151} and \cite[Thm 7.2]{DvVe2}:

\begin{theorem}\label{th:2.1}
For any PMV-algebra $(M;\oplus,^-,^\sim,0,1)$,
there exists a unique $($up to iso\-morphism$)$ unital $\ell$-group
$G$ with a strong unit $u$ such that $(M;\oplus,^-,^\sim,0,1) \cong (\Gamma(G,u);\oplus,^-,^\sim,0,u)$. The
functor $\Gamma$ defines a categorical equivalence of the variety
of PMV-algebras with the category of unital $\ell$-groups.
\end{theorem}

\begin{theorem}\label{th:2.2}
For every PEA $(E;+,0,1)$ with \RDP$_1,$ there is a unique unital po-group $(G,u)$ with \RDP$_1$\ such that $(E;+,0,1) \cong (\Gamma(G,u);+,0,u).$ The
functor $\Gamma$ defines a categorical equivalence of the category
of PEAs with the category of unital po-groups with \RDP$_1.$
\end{theorem}

In \cite[Thm 8.3, 8.4]{DvVe2}, it was proved that if $(M; \oplus,^-,^\sim,0,1)$ is a PMV-algebra, then $(M;+,0,1),$ where $+$ is defined by (2.1), is a pseudo effect algebra with RDP$_2.$  Conversely, if $(E; +,0,1)$ is a pseudo effect algebra with RDP$_2,$ then $E$ is a lattice, and by \cite[Thm 8.8]{DvVe2}, $(E; \oplus,^-,^\sim,0,1),$ where
$$
a\oplus b := (b^-\minusli (a\wedge b^-))^\sim,\eqno(2.2)
$$
is a PMV-algebra. In addition, a PEA $E$ has RDP$_2$ iff $E$ is a lattice and $E$ satisfies RDP$_1,$ see \cite[Thm 8.8]{DvVe2}.

We note that if $M$ is a PMV-algebra, then the notion of an ideal of an PMV-algebra $M$ coincides with the notion of an ideal taken in the PEA $M$  with $+$ defined by (2.1).

Let $A$ and $B$ be two non-void subsets of a PMV-algebra $M,$ we set (i) $A\oplus B:=\{a\oplus b: a \in A, b \in B\},$ (ii) $A + B = \{a+b: $ if $a+b$ exists in $M$ for $a\in A,\ b \in B\}.$ We say that $A+B$ is {\it defined} in $M$ if $a+b$ exists in $M$  for any $a \in A$ and any $b \in B.$

An ideal $I$ of $M$ is normal if $x\oplus I=I\oplus x$ for any $x \in M;$ let $\mathcal N(M)$ be the set of normal ideals of $M.$  There is a one-to-one correspondence between normal ideals and congruences for PMV-algebras, \cite[Thm 3.8]{GeIo}. The quotient PMV-algebra over a normal ideal $I,$  $M/I,$ is defined as the set of all elements of the form $x/I := \{y \in M :
x\odot y^- \oplus y\odot x^- \in I\},$ or equivalently, $x/I := \{y \in M : x^\sim \odot y \oplus y^\sim \odot x \in I\}.$

We can define a maximal ideal of a PEA $E$ in the same way as for PMV-algebras, and an ideal $I$ of $M$ is {\it normal} if $x+I=I+x$ for any $x\in M.$ We note the normality of an ideal of an PMV-algebra $M$ is the same as that for the PEA $M$ with $+$ determined by (2.1).

We define (i) the {\it radical} of a PMV-algebra $M$,
$\mbox{Rad}(M),$ as the set
$$
\mbox{\rm Rad}(M) = \bigcap\{I:\ I \in {\mathcal M}(M) \},
$$
and (ii) the {\it normal radical} of $M$, $\mbox{Rad}_n(M)$, via
$$
\mbox{Rad}_n(M) = \bigcap\{I:\ I \in {\mathcal N}(M) \cap {\mathcal M}(M)\}.
$$

By \cite[Prop. 4.1, Thm 4.2]{DDJ}, it is possible to show that

$$
\mbox{\rm Rad}(M) \subseteq \mbox{\rm Infinit}(M) \subseteq
\mbox{\rm Rad}_n(M).\eqno(2.3)
$$

Finally, we say that a mapping $s: M\to [0,1]$ is a {\it state} on a PMV-algebra $M$ if (i) $s(1)=1,$ and (ii) $s(a+b)=s(a)+s(b)$ whenever $a+b$ is defined in $M.$ We say that a state $s$ is {\it extremal} if from $s=\lambda s_1 +(1-\lambda)s_2,$ where $s_1, s_2$ are states on $M$ and $\lambda$ is a real number such that $0<\lambda <1,$ it follows $s=s_1=s_2.$ We denote by $\mathcal S(M)$ and $\partial_e\mathcal S(M)$ the set of all states and the set of all extremal states on $M,$ respectively. If $M$ is an MV-algebra, $\mathcal S(M)$ is always a non-void set. But if $M$ is a PMV-algebra, it can happen that $M$ is stateless, see e.g. \cite{DDJ, 156, DvHo}. The set $\Ker(s)=\{a \in M: s(a)=0\},$ the {\it kernel} of $s,$ is a normal ideal. A state $s$ is extremal iff $\Ker(s)$ is a maximal ideal, and conversely, every maximal and normal ideal is a kernel of a unique extremal state, see \cite{156}. In addition, a state $s$ is extremal iff $s(a\wedge b) = \min\{s(a),s(b)\},$ $a,b \in M,$ \cite[Prop 4.7]{156}.

A state on a unital $\ell$-group $(G,u)$ is a mapping $s: G \to \mathbb R$ such that (i) $s(G^+) \subseteq \mathbb R^+,$ (ii) $s(g_1+g_2) = s(g_1)+ s(g_2)$ for all $g_1,g_2 \in G,$ and (iii) $s(u)=1.$  There is a one-to-one correspondence between the states on $(G,u)$ and $\Gamma(G,u);$ every state on $\Gamma(G,u)$ can be extended to a unique state on $(G,u),$ see \cite{156}.

\section{$\mathbb H$-perfect PMV-algebras}

From this section, $\mathbb H$ will denote a subgroup of the group of real numbers $\mathbb R$ such that $1 \in \mathbb H.$ The main aim of this section is to introduce and study PMV-algebras which can be split into a family of comparable slices indexed by the elements of the subgroup $\mathbb H.$ Such prototypical examples are PMV-algebras represented in the form

$$\Gamma(\mathbb H \lex G,(1,0)),\eqno(3.1)
$$
where $G$ is any $\ell$-group, and $\mathbb H\lex G$ denotes the {\it lexicographic product} of $\mathbb H$ with $G;$ we note that in such a lexicographic product, the order $\le$ is defined as follows: $(h_1,g_1)\le (h_2,g_2)$ iff either $h_1 < h_2$ or $h_1=h_2$ and $g_1\le g_2.$ It is clear that the element $u=(1,0)$ is a strong unit for $\mathbb H \lex G$ and (3.1) defines a PMV-algebra.

A very special case is when $G=O,$ where $O$ is the zero $\ell$-group, because then $\Gamma(\mathbb H\lex O,(1,0))$ is isomorphic to the Archimedean MV-algebra $\Gamma(\mathbb Z,1).$ In general, if $G \ne O,$ (3.1) does not give an Archimedean PMV-algebra.

By $\mathbb Q$ we denote the group of rational numbers in $\mathbb R,$ $\mathbb Z$ denotes the group of integers, and given an integer $n\ge 1,$ $\frac{1}{n}\mathbb Z:=\{\frac{i}{n}: i \in \mathbb Z\}.$ By \cite[Lem 4.21]{Goo}, every $\mathbb H$ is either {\it cyclic}, i.e. $\mathbb H=\frac{1}{n}\mathbb Z$ for some $n \ge 1$ or $\mathbb H$ is dense in $\mathbb R.$

For example, if $\mathbb H= \mathbb H(\alpha)$ is a subgroup of $\mathbb R$ generated by $\alpha \in [0,1]$ and $1,$ then $\mathbb H = \frac{1}{n}\mathbb Z$ for some integer $n\ge 1$ if $\alpha$ is a rational number. Otherwise, $\mathbb H(\alpha)$ is countable and dense in $\mathbb R,$ and $M(\alpha):= \Gamma(\mathbb H(\alpha),1)=\{m+n\alpha: m,n \in \mathbb Z,\ 0\le m+n\alpha \le 1\},$ see \cite[p. 149]{CDM}. In addition, $\{\mathbb H(\alpha): \alpha \in(0,1)\}$ is an uncountable system of non-isomorphic subgroups of $\mathbb R.$

We set $[0,1]_\mathbb H:=[0,1]\cap \mathbb H.$

\begin{definition}\label{de:3.1}
{\rm We say that a PMV-algebra $M$ is $\mathbb H$-{\it perfect}, if there is a system $(M_t: t \in [0,1]_\mathbb H)$ of nonempty subsets of $M$ such that it is an $\mathbb H$-{\it decomposition} of $M,$ i.e. $M_s \cap M_t= \emptyset$ for $s<t,$ $s,t \in [0,1]_\mathbb H$ and $\bigcup_{t\in [0,1]_\mathbb H} = M$ and \begin{enumerate}

\item[{\rm (a)}]
$M_s \leqslant M_t$ for all $ s<t$, $s,t \in [0,1]_\mathbb H,$

\item[{\rm (b)}]
$M_t^- = M_{1-t} =M_t^\sim$ for any $t \in [0,1]_\mathbb H.$

\item[{\rm (c)}]  if $x \in M_v$ and $y \in M_t$, then
$x\oplus y \in M_{v\oplus t},$ where $v\oplus t=\min\{v+t,1\}.$

\end{enumerate}}
\end{definition}

We recall that if $\mathbb H=\frac{1}{n}\mathbb Z,$ a $\frac{1}{n}\mathbb Z$-perfect PMV-algebra is said to be $n$-{\it perfect}, for more details on $n$-perfect PMV-algebras, see \cite{Dv08}.

For example, let $M =\Gamma(\mathbb H\lex G, (1,0)).$ We set $M_0=\{(0,g): g \in G^+\},$ $M_1:=\{(1,-g): g \in G^+\}$ and
for  $t \in [0,1]_\mathbb H\setminus \{0,1\},$ we define $M_t:=\{(t,g): g \in G\}.$ Then $(M_t: t \in [0,1]_\mathbb H)$ is an $\mathbb H$-decomposition of $M$ and $M$ is an $\mathbb H$-perfect PMV-algebra.

Sometimes we will write also $M=(M_t: t \in [0,1]_\mathbb H)$ for $\mathbb H$-perfect PMV-algebras.

We say that a state $s$ on a PMV-algebra $M$ is an $\mathbb H$-{\it valued state} if $s(M)=\mathbb H.$   If $s(M) \subseteq [0,1]_\mathbb H,$ we say that $s$ is an $\mathbb H$-{\it state}.  In particular, if $\mathbb H = \frac{1}{n}\mathbb Z,$  a $\frac{1}{n}\mathbb Z$-valued state is also said to be an $(n+1)$-{\it valued discrete state}, \cite{DXY}.

The basic properties of $\mathbb H$-perfect PMV-algebras are described as follows.

\begin{theorem}\label{th:3.2}
Let $M =(M_t: t \in [0,1]_\mathbb H)$ be an  $\mathbb H$-perfect PMV-algebra.

\begin{enumerate}

\item[{\rm (i)}] Let $a \in M_v,$  $b \in M_t$. If $v+t < 1$, then
$ a+b $ is defined in $M$ and $a+b \in M_{v+t}$; if $a+b$ is defined
in $M$, then $v+t\le 1.$

\item[{\rm (ii)}]
$M_v + M_t$ is defined in $M$ and $M_v + M_t = M_{v+t}$ whenever $v+t < 1.$

\item[{\rm (iii)}] If $a \in M_v$ and $b \in M_t$, and
$v+t > 1$, then $a+b$ is not defined in $M.$

\item[{\rm (iv)}]
$M$ admits a unique state. This state is an $\mathbb H$-valued state such that
$s(M_t) = \{t\}$ for each $t \in [0,1]_\mathbb H.$ Then $M_t =
s^{-1}(\{t\})$ for any $t \in [0,1]_\mathbb H,$ and $s$ is an $\mathbb H$-valued state such that $\Ker(s)=M_0.$

\item[{\rm (v)}]  $M_0$ is a normal and maximal ideal
of $M$ such that $M_0 + M_0 = M_0.$

\item[{\rm (vi)}] $M_0$ is a unique maximal ideal of
$M$,   and $M_0= \Rad(M) = \Infinit(M).$

\item[{\rm (vii)}]
Let $M = (M'_t: t \in [0,1]_\mathbb H)$
be another representation of $M$ satisfying {\rm (a)--(c)} of Definition {\rm \ref{de:3.1}}, then
$M_t = M_t'$ for each $t \in [0,1]_\mathbb H.$

\item[{\rm (viii)}] The quotient PMV-algebra $M/M_0\cong \Gamma(\mathbb H,1).$

\end{enumerate}
\end{theorem}

\begin{proof}
(i)  Assume $a\in M_v$ and $b \in M_t$ for
$v+t< 1$. Then $b^- \in M_{1-t}$, so that $a \le b^-$, and $a+b$ is
defined in $M.$  Conversely, let $a+b$ be defined, then $a\le b^-\in
M_{1-t}$ which gives $v+t\le 1.$

(ii) By (i), we have  $M_v +M_t \subseteq M_{v+t}$.
Suppose $z \in M_{v+t}$.  Then, for any $x\in M_v,$ we have $x \le z$,
and hence $y = z\minusli x$ is defined in $M$, and $y \in M_w$ for
some $w \in [0,1]_\mathbb H.$ Since $z = y+x \in M_{v+t}\cap M_{v+w}$,  we conclude $t=w$ and $M_{v+t} \subseteq M_v + M_t.$

(iii)  If $a+b \in M$, then $a \le b^- \in M_{1-t} \leqslant M_v$ which
gives $a\le b^- \le a$, that is, $a = b^-$. This is possible only if
$v = 1-t$ which is impossible.

(iv)--(vi) Define a mapping $s:\ M \to [0,1]$ by $s(x) = t$ if $x \in M_t$.  It is clear that $s$ is a well-defined mapping.  Take $a,b \in M$ such
that $a+ b$ is defined in $M$.  Then there are unique indices $v$ and $t$ such that $a\in M_v$ and $b \in M_t$. By (i), $v+t \le 1$
and  $a+b \in M_{v+t}$. Therefore, $s(a+b) = v+t  = s(a) +
s(b).$ It is evident that $s(1) = 1,$ $\Ker(s)=M_0,$ and $M_t =s^{-1}(\{t\})$ for $t \in [0,1]_\mathbb H.$ In particular, $M_0$ is a normal ideal of $M.$

\vspace{2mm}
{\it Maximality of $M_0.$} Take $x \in M_t \setminus M_0,$ where $0 < t <1,$ $t \in [0,1]_\mathbb H.$ Let $I$ be an ideal of $M$ generated by $M_0$ and $x.$
Then, for every $v <t,$ $s \in [0,1]_\mathbb H,$ we have $M_v \leqslant  M_t,$ whence $M_v \subseteq I.$ There are two cases: (a) there is no $v\in [0,1]_\mathbb H$ such that $0<v<t.$ Then $t= 1/n$ for some integer $n\ge 1$ and $\mathbb H =\frac{1}{n}\mathbb Z.$  If $n=1$, then $s(x)=1,$  $s(x^-) =0,$ and $x^- \in M_0.$ Hence, $1 \in I.$

If $n \ge 2,$ then $y:=(n-1)x$ is defined in $M,$ and $y \in I.$ For the element $y^-,$ we have $s(y^-) = 1/n,$ so that $y^- \in I$ which means $1 \in I.$

(b) $\mathbb H$ is no cyclic subgroup of $\mathbb R,$ so that it is dense in $\mathbb R.$ There is a strictly decreasing sequence $\{t_i\}$ of non-zero elements of $[0,1]_\mathbb H$ such that $t_i\searrow 0.$ For every $t_i,$ there is a maximal integer $m_i$ such that $y_i:=m_i t_i$ is defined in $M.$ Hence, for enough small $t_i$, $s(y_i^-) <t$ so that $y_i^- \in I$ which again proves $I=M,$ and $M_0$ is a maximal ideal.

\vspace{2mm}
{\it Uniqueness of a maximal ideal.} Assume that $I$ is another maximal ideal of $M.$  Let there be $x \in M_t \cap I$ for some $t \in [0,1]_\mathbb H,$ $t>0.$ Then, for every $z \in M_0,$ we have $z \le x$ and $z \in I,$ so that $M_0 \subseteq I.$ The maximality of $M_0$ yields $M_0=I.$

Since there is a one-to-one correspondence between extremal states and maximal ideals which are also normal given by $s \leftrightarrow \Ker(s),$ \cite[Prop 4.3-4.6]{156}, we see that $M$ has a unique state, this state is extremal and an $\mathbb H$-valued state.

\vspace{2mm}
Finally, we show $M_0 =\Infinit(M).$  Since $M_0+M_0=M_0,$ we have $M_0 \subseteq \Infinit(M).$ Let $x \in \Infinit(M).$ Then $mx$ exists in $M$ for any integer $m\ge 1.$ Hence, $s(mx)=ms(x)\le 1$ which gives $s(x)=0$ and $x \in \Ker(s)=M_0.$ From (2.3), we conclude $M_0= \Rad(M) = \Infinit(M).$

(vii)  If $M = (M_t': t \in [0,1]_\mathbb H)$ is another representation of $M$, then by (iv), $M$ admits a state $s'$ such that $M_t' =
s'^{-1}(\{t\})$ for any $t \in [0,1]_\mathbb H.$  Since $M$ admits a unique state, $s = s'$ and $M_t = M_t'$ for each $t \in [0,1]_\mathbb H.$

(viii)  By (iv), there is a (unique extremal) state $s$ on $M$ such that $\Ker(s)=M_0.$ Then $a \sim b$ iff $s(a)=s(a\wedge b)=s(b).$ Since $s$ is an extremal state, $\Ker(s)$ is a maximal ideal and normal. Hence, $M/\Ker(s) =[0,1]_\mathbb H=\Gamma(\mathbb H,1).$
Hence, $M/M_0 \cong \Gamma(\mathbb H,1).$
\end{proof}

In the rest of this section, we will study some varieties of PMV-algebras generated by $\mathbb H$-perfect PMV-algebras. We show that there are two important cases depending on whether $\mathbb H$ is a cyclic or non-cyclic subgroup of $\mathbb R.$ We note that the cyclic case was studied in \cite{225}.

If $\mathcal K$ is a family of PMV-algebras, we denote by $\mathcal{V(K)}$ the variety of PMV-algebras generated by $\mathcal K.$ If $\mathcal K=\{K\},$ we denote simply $\mathcal V(K):=\mathcal{V(K)}.$

To show these varieties, we introduce so-called top varieties of PMV-algebras, see \cite{DvHo}. The basic tool in our considerations is Theorem \ref{th:2.1}. In particular, it entails a one-to-one correspondence between the
set of ideals, normal ideals, maximal ideals of $M = \Gamma(G,u)$,
and the set of convex $\ell$-subgroups, ${\mathcal C}(G)$,
$\ell$-ideals, ${\mathcal L}(G)$, and maximal convex $\ell$-subgroups,
${\mathcal M}(G)$, of $(G,u)$, see \cite{156}; the one-to-one mapping
$\psi:\ {\mathcal I}(M) \to {\mathcal C}(G)$ is defined by
$$
\psi(I)= \{x\in G:\ \exists \ x_i, y_j \in I,\ \ x = x_1 +\cdots +
x_n - y_1- \cdots - y_m\}.\eqno(3.2)
$$

Let $M = \Gamma(G,u)$ be a PMV-algebra, where $(G,u)$ is a unital
$\ell$-group. By a {\it value} of $u$ in $(G,u)$ we mean a convex
$\ell$-subgroup  $H$ of $(G,u)$ maximal under condition $H$ does not
contain $u$. Hence, $\psi^{-1}(H)$ is a maximal ideal of $M$, where
$\psi$ is defined by (3.2),  and vice versa. If $I$ is a maximal
ideal of $M$, then $\psi(I)$ is a value of $u$ in $(G,u)$.

For any value $V$ of $(G,u)$, we set
$$K(V) = \bigcap_{g \in G} g^{-1}Vg
$$
(for a moment we use a multiplicative form of $(G,u)$).
 Then $K(V)$ is a normal convex $\ell$-subgroup of $(G,u)$
contained in $V$, and $(G/K(V), G/V)$ is a primitive transitive
$\ell$-permutation group called a {\it top component} of $G$.

Let ${\mathcal  V}$ be a variety of PMV-algebras and let
$\Gamma^{-1}({\mathcal  V}) =\{(G,u):\ \Gamma(G,u) \in {\mathcal  V}\}.$ We
recall that ${\mathcal  V}$ contains  a trivial  PMV-algebra (i.e.
$0=1$). Then by \cite[Thm 3.1]{DvHo}, $\Gamma^{-1}({\mathcal  V})$ is an
equational class of unital $\ell$-groups in some extended sense:
$\Gamma^{-1}({\mathcal  V})$ is not a variety in the usual sense of
universal algebra, but rather a class of unital $\ell$-groups
described by equations in the language of unital $\ell$-groups.

Let
$$ {\mathcal  T}({\mathcal  V}) =\{\Gamma(G,u):\ \Gamma(G/K(V),u/K(V)) \in
{\mathcal  V},\ V\in {\mathcal  M}(G) \} \cup \{\{0\}\}.\eqno(3.3)
$$
By  \cite[Cor. 4.3]{DvHo}, ${\mathcal  T}({\mathcal  V})$ is a variety, we
call it a {\it top variety} of ${\mathcal  V}.$

We denote by ${\mathcal  M}$   the set of PMV-algebras $M$ such that
either every maximal ideal of $M$ is normal or $M$ is trivial. In
\cite[(6.1)]{DDT}, there was shown that ${\mathcal  M}$ is a variety such
that

$$ {\mathcal  M} = {\mathcal  T}(\mathcal  {MV}) = {\mathcal  T}({\mathcal  N}) = {\mathcal  T}({\mathcal
M}),\eqno(3.4)
$$
where $\mathcal{MV},$ as it was already mentioned,  is the variety of MV-algebras and $\mathcal  N$ is the
set of normal-valued PMV-algebras, which according to \cite[Thm 6.8]{156}
is a variety. (We recall that a {\it value} of any non-zero element $b\in M$ is any ideal $I$ of $M$ maximal under the condition $b \not\in I.$ The ideal $I^*$ generated by $I$ and $b$ is said to be a {\it cover} of $I$ and we say that $I$ is normal in its cover if $x\oplus I=I\oplus x$ for any $x\in I^*.$ Finally, we say that $M$ is {\it normal-valued} if every value is normal in its cover.)

We recall that according to Theorem \ref{th:2.1}, it is possible to show that a PMV-algebra $M=\Gamma(G,u)$ is symmetric iff $u \in C(G),$ \cite[p. 98]{187}.

Let $\mathbb H$ be a subgroup of $\mathbb R$ such that $1\in \mathbb H.$
We define $\mathcal{PPMV}_\mathbb H$, the system of $\mathbb H$-perfect PMV-algebras
($\mathcal{PPMV}_\mathbb H^S$ symmetric $\mathbb H$-perfect PMV-algebras),
${\mathcal V}(\mathcal{PPMV}_\mathbb H)$, the variety generated by all $\mathbb H$-perfect PMV-algebras,  and  $\mathcal {BP}_\mathbb H$ (and $\mathcal{SBP}_\mathbb H$), the system
of (symmetric) PMV-algebras $M$ such that either every maximal ideal
of $M$ is normal and every extremal state of $M$ an $\mathbb H$-state or $M$ is the one-element PMV-algebra. Or equivalently, either every maximal ideal $I$ of $M$ is normal and $M/I$ is a subalgebra of $\Gamma(\mathbb H,1).$

If $\mathbb H=\frac{1}{n}\mathbb Z,$ instead of $\mathcal{PPMV}_\mathbb H,$ $\mathcal{BP}_\mathbb H$ and $\mathcal{SBP}_\mathbb H,$ we write according to \cite{Dv08}, $\mathcal{PPMV}_n,$ $\mathcal{BP}_n$ and $\mathcal{SBP}_n,$ respectively.

In such a case, $\mathcal{BP}_n$ consists of all PMV-algebras $M$ such that every maximal ideal is normal and every extremal state is  $(k+1)$-valued,
where $k$ divides $n,$ or $M$ is the one-element PMV-algebra. Or
equivalently, either every maximal ideal $I$ of $M$ is normal and
$M/I \cong \Gamma(\mathbb Z,k)$ where $k|n,$ or $M=\{0\}$.  It is
clear that $\mathcal{BP}_1 = \mathcal {BP},$ and $\mathcal {SBP}_1 = \mathcal
{SBP},$ where $\mathcal {BP}$ and $\mathcal {SBP}$ were studied in
\cite{DDT}. We have $\mathcal {BP}_m \subseteq \mathcal {BP}_n$ iff $m|n$.
If $n$ is prime, then $\mathcal {BP}_n$ is of particular interest.

In  \cite[Cor. 11]{DiLe2}, there is  presented a characterization of
MV-algebras which are members of the variety ${\mathcal V}({\mathcal
M}_n(\mathbb Z))$, that is, the variety generated by the MV-algebra
$\Gamma(\frac{1}{n}\mathbb Z\lex   \mathbb Z, (1,0))$. They showed that the
variety ${\mathcal V}({\mathcal M}_n(\mathbb Z))$ is characterized
by the following identities

$$ ((n+1)\odot x^n)^2 = 2 \odot x^{n+1}, \eqno(3.5)
$$
$$
(p\odot x^{p-1})^{n+1}= (n+1)\odot x^p,  \eqno(3.6)
$$
for every integer $p$,  $1<p< n$, such that $p$ is not a divisor of
$n.$

These identities were used to describe the following varieties. Let $\mathcal V_{P_n}$ and $\mathcal V^S_{P_n}$ be the varieties of PMV-algebras and symmetric PMV-algebras, respectively, satisfying the identities (3.5)--(3.6). Then the following result was established in \cite[Thm 5.1]{Dv08}.

\begin{theorem}\label{th:5.1}
We have
${\mathcal T}({\mathcal V}_{P_n}) = \mathcal{
BP}_n,$ and $\mathcal {BP}_n$ is a variety such that ${\mathcal T}(\mathcal{
BP}_n) = \mathcal {BP}_n = {\mathcal T}({\mathcal V}(\Gamma(\mathbb Z,n)))
={\mathcal T}({\mathcal V}({\mathcal M}_n(\mathbb Z))).$
\end{theorem}

For the case that $\mathbb H$ is not cyclic, we extend Theorem \ref{th:5.1} as follows. We note that by (3.3) we can define $\mathcal{T(V)}$ for any family $\mathcal V$ of PMV-algebras (not only for varieties).

\begin{theorem}\label{th:5.2}
Let $\mathbb H$ be not a cyclic subgroup of $\mathbb R.$ Then $\mathcal{T}(\mathcal{BP}_\mathbb H) = \mathcal{BP}_\mathbb H$ and $\mathcal{BP}_\mathbb H.$
In addition, $\mathcal{T(V(BP}_\mathbb H))=\mathcal M= \mathcal {T(V(PPMV}_\mathbb H)).$
\end{theorem}

\begin{proof}
By the definition of $\mathcal{BP}_\mathbb H,$ we have $\mathcal { BP}_\mathbb H \subset \mathcal {
M}.$  Due to (3.4), $\mathcal { BP}_\mathbb H \subseteq \mathcal { T}(\mathcal { BP}_\mathbb H)
\subseteq \mathcal { M}.$ Let $M\in \mathcal { T}(\mathcal { BP}_\mathbb H)$ and let $I$ be
a maximal ideal of $M$. Then $I$ is normal and $M/I \in \mathcal {
BP}_\mathbb H$. Since $I$ is maximal, $M/I$ is an MV-subalgebra of
$\Gamma(\mathbb H,1)\subseteq \Gamma(\mathbb R,1)$ and $M/I$ has a unique maximal ideal, $J$,
which is the zero one. Therefore,  $M/I \cong  (M/I)/J \in \mathcal {
BP}_\mathbb H.$  This proves that
$\mathcal {BP}_\mathbb H= \mathcal { T}(\mathcal { BP}_\mathbb H).$

By (iv) of Theorem \ref{th:3.2}, we have $\mathcal{PPMV}_\mathbb H \subseteq \mathcal{BP}_\mathbb H \subseteq \mathcal M.$ Then $\mathcal{V(PPMV}_\mathbb H) \subseteq \mathcal{V(BP}_\mathbb H) \subseteq \mathcal M.$ It is clear that $\Gamma(\mathbb H,1) \in \mathcal{V(PPMV}_\mathbb H).$ Since $H$ is dense in $\mathbb R,$ by \cite[Prop 8.1.1]{CDM}, $\mathcal{MV} =\mathcal V(\Gamma(\mathbb H,1))$ and, therefore by (3.4), $\mathcal {M} = \mathcal{T(MV)} \subseteq
\mathcal{T(V(PPMV}_\mathbb H)) \subseteq \mathcal{T(V(BP}_\mathbb H)) \subseteq \mathcal {T(M)} =\mathcal M.$
\end{proof}

We note that according to Theorem \ref{th:5.1}, if $\mathbb H$ is cyclic, then $\mathcal{BP}_\mathbb H$ is a variety. 
In the next theorem, we show that if $\mathbb H\ne \mathbb R$ is not cyclic, then $\mathcal{BP}_\mathbb H$ is not a variety.

Now we show when $\mathcal{BP}_\mathbb H$ is a variety.

\begin{theorem}\label{th:5.3}
The systems $\mathcal{BP}_\mathbb H$ and $\mathcal{SBP}_\mathbb H$ are varieties if and only if either $\mathbb H$ is cyclic or $\mathbb H =\mathbb R.$ In such a case,  $\mathcal{BP}_\mathbb R=\mathcal M,$  $\mathcal{SBP}_\mathbb R=\mathcal{SYM}\cap \mathcal M,$ and all $\mathcal{BP}_n \ne \mathcal{M}$ are mutually different.
\end{theorem}

\begin{proof}
The case when $\mathcal {BP}_\mathbb H$ is a variety for $\mathbb H=\frac{1}{n}\mathbb Z$  was shown in Theorem \ref{th:5.1}. If $\mathbb H=\mathbb R,$ then evidently $\mathcal{BP}_\mathbb R \subseteq \mathcal M$ and if $M \in \mathcal M$, then every its maximal ideal $I$ is normal, and $M/I$ is a subalgebra of $\Gamma(\mathbb R,1),$ so that $M \in \mathcal {BP}_\mathbb R.$

Now assume that $\mathbb H$ is not a cyclic subgroup of $\mathbb R$ and let $\mathbb H \ne \mathbb R.$ If $\mathcal{BP}_\mathbb H$ is a variety, by Theorem \ref{th:5.2}, $\mathcal{BP}_\mathbb H = \mathcal{T(BP}_\mathbb H)$ and $\mathcal{MV} \subseteq \mathcal{BP}_\mathbb H$ so that $M=\Gamma(\mathbb R,1) \in \mathcal{MV}\subseteq \mathcal{BP}_\mathbb H,$ but on the other hand, $M$ does not belong to $\mathcal{BP}_\mathbb H$ by definition of $\mathcal{BP}_\mathbb H$ because $\mathbb R$ is not a subgroup of $\mathbb H.$

In a similar way we deal with $\mathcal{SBP}_\mathbb H.$
\end{proof}

In what follows, we describe subdirectly irreducible elements in $\mathcal{BP}_\mathbb H,$ Theorem \ref{th:5.5}. It will be shown that they are only $\mathbb K$-perfect PMV-algebras, where $\mathbb K$ is a subgroup of $\mathbb H$ such that $1 \in \mathbb K.$

If $A$ is a subset of a PMV-algebra $M,$ we denote by $\langle A\rangle$ the subalgebra of $M$ generated by $A.$

\begin{proposition}\label{pr:5.4}
{\rm (1)} Let $M$  be a
PMV-algebra such that $\mathcal S(M)\ne \emptyset,$ and let us define
$$
M_t'=\bigcap \{s^{-1}(\{t\}): \ s \in
\partial_e{\mathcal S}(M)\},\quad t \in [0,1]_\mathbb H.
$$
Then $$ \langle \bigcup_{t \in [0,1]_\mathbb H} M_t' \rangle = \bigcup_{t \in [0,1]_\mathbb H}M_t'.
$$

{\rm (2)} If $M \in {\mathcal M}$, then $\bigcup_{t \in [0,1]_\mathbb H}M_t'$ is the
biggest subalgebra of $M$ having a unique extremal state, and this
state is an $\mathbb H$-state.


\end{proposition}

\begin{proof} (1)  It is clear that $M':=
\bigcup_{t \in [0,1]_\mathbb H} M_t'$ contains $0,1$, and if $x\in M_t'$, then
$x^-,x^\sim \in M_{1-t}'$,  \cite[Prop. 4.1]{156}. If $x\in M_v'$
and $y \in M_t'$, then $x\oplus y \in M_{v\oplus t}'.$

(2) If $s_1$ and $s_2$ are  extremal states on $M$, then their
restrictions to $M'$ are extremal states on $M'$ which are $\mathbb H$-states, and $s_1(a) = s_2(a)$ for any $a \in M'.$
Conversely, if $s$ is an extremal state on $M'$, then there is an
extremal state $\hat s$ on $M$ such that $\mbox{Ker}(s) =
\mbox{Ker}(\hat s)\cap M'$. Then $\mbox{Ker}(s) =
\mbox{Ker}(s_{|M'})$ which yields $s = \hat s_{|M'}$. Therefore, $s
= {s_1}_{|M'}$ for any extremal state $s_1$ on $M$.  Let $s'$ be the
unique extremal state on $M'$, then $M_t' = s'^{-1}(\{t\})$
whenever $M_t'\ne \emptyset$ for any $t \in [0,1]_\mathbb H.$

Let now $M''$ be an arbitrary subalgebra of $M$ having a unique
extremal state $s''$, and let this state be an $\mathbb H$-state.
Since every restriction of an extremal state of $M$ to $M''$ is an
extremal state on $M''$, and any extremal state on $M''$ can be
extended to an extremal state on $M$, we see that $s''^{-1}(\{t\})
\subseteq M_t'$ for any $t \in [0,1]_\mathbb H,$ hence, $M'' \subseteq M'.$
\end{proof}

The following characterization of subdirectly irreducible elements was originally proved in \cite[Lem 5.3]{Dv08} for the case $\mathbb H=\frac{1}{n}\mathbb Z.$ In the following lemma we extend it for a general case of $\mathbb H.$ Nevertheless the proof for our case follows the same ideas as that in \cite{Dv08}, to be self-contained, we present the proof if full completeness together with necessary changes.

\begin{theorem}\label{th:5.5}
If $M \in \mathcal{BP}_\mathbb H$  ($M \in \mathcal{ SBP}_\mathbb H$)
is subdirectly irreducible, then either $M$ is trivial or $M =
\bigcup_{t \in [0,1]_\mathbb H} M_t,$ where  $M_t =\bigcap \{s^{-1}(\{t\}): \ s
\in \partial_e\mathcal{S}(M)\}$ for each $t \in [0,1]_\mathbb H,$
$|\partial_e\mathcal{ S}(M)|= 1,$ and $M$ is a $\mathbb K$-perfect PMV-algebra (symmetric and
$\mathbb K$-perfect PMV-algebra),  where $\mathbb K$ is a subgroup of $\mathbb H$ such that $1 \in \mathbb K.$
\end{theorem}

\begin{proof}
Assume $M = \Gamma(G,u)$ for a unital
$\ell$-group $(G,u)$ is non-trivial.  Due to Theorem \ref{th:2.1},
$M$ is subdirectly irreducible iff $G$ is subdirectly irreducible.
In view of \cite[Cor. 7.1.3]{Gla}, $G$ has a faithful transitive
representation. Therefore, by \cite[Cor. 7.1.1]{Gla}, this is
possible iff there is a prime subgroup $C$ of $G$ such that
$\bigcap_{g \in G} g^{-1}Cg = \{1\}$ (we use the multiplicative form
of $(G,u)$). In such a case, the set $\Omega :=\{Cg:\ g \in G\}$ of
right cosets of $C$ is totally ordered  assuming $Cg \le Ch$ iff $g
\le ch$ for some $c \in C,$ and $G$ has a faithful transitive
representation on $\Omega$, namely $\psi(f) = Cgf$, $f \in G$, with
$\mbox{Ker}(\psi) = \bigcap_{g \in G} g^{-1}Cg = \{1\}.$

Since the system of prime subgroups of $G$ forms a root system,
there is a unique maximal ideal $I$ of $M$ such that $C \subseteq
\psi(I) =: \hat I$, where $\psi(I)$ is defined by (3.2).

(I) Assume $M/I \cong \Gamma(\mathbb Z,n).$ Due to the one-to-one
correspondence between normal and maximal ideals, $I,$ and extremal states, $s,$
given by $I=\Ker(s)$, let the maximal ideal $I$ correspond to a unique
extremal state, say $s_I$. We define $I_t = s_I^{-1}(\{t\})$ for
any $t \in [0,1]_\mathbb H.$  Then $M = \bigcup_{t \in [0,1]_\mathbb H} I_t.$

\vspace{2mm}\noindent {\it Claim 1.} {\it If $a \in I$ and $b \notin
I$, then $a \le b.$} \vspace{2mm}

There are two possibilities: (1) $Cg = Cg(a\wedge b)$ and (2) $Cg
\ne Cg(a\wedge b)$.

(1)  Let $Cg = Cg(a\wedge b)$. Then $a\wedge b \in g^{-1}Cg
\subseteq g^{-1}\hat I g = \hat I.$ Because $g^{-1}Cg$ is also
prime, we have $a \in g^{-1}Cg$.  Hence, $Cga = Cg$, i.e., $Cga = Cg
= Cg(a\wedge b) \le Cgb.$

(2)  Let  $Cg \ne Cg(a\wedge b)$. The transitivity of $G$ entails
there is an $h \in G$ such that $Cgh = Cg(a\wedge b)$. Then $gh =
cg(a\wedge b)$ for some $c \in C$, and $h = g^{-1}cg (a\wedge b) \in
\hat I.$ Hence, $Cgh = Cghh^{-1}(a\wedge b)$ and $h^{-1}(a\wedge b)
= (h^{-1} a) \wedge (h^{-1}b) \in (gh)^{-1}C(gh).$ Since
$(gh)^{-1}C(gh)$ is prime, and $h \in \hat I$, we get $h^{-1}a \in
(gh)^{-1}C(gh)$.  Then $h^{-1}a = (gh)^{-1}cgh$ for some $c \in C$,
and $ga = ghh^{-1} a = cgh$, i.e., $Cga = Cgh$. But $Cga= Cgh =
Cg(a\wedge b) \le Cgb$.

Combining (1) and (2), we get $Cga \le Cgb$ for any $g \in G,$ i.e.,
$a \le x\wedge b \le x$, and $a = a\wedge b$ proving Claim 1.

\vspace{2mm}\noindent {\it Claim 2.} {\it If $s$ is an arbitrary
extremal state on $M$,  $s(x) = s_I(x)$ for any $x\in
I.$}\vspace{2mm}

Let $x \in I =\mbox{Ker}(s_I)$, then by Claim 1, $x\le x^-$ and
$k\odot x \le (k\odot x)^-$ for any integer $k\ge 1.$  We assert
that $s(x) = 0$.  If not, then $s(x) =t$ for some $t \in [0,1]_\mathbb H.$  Hence, $1 = s(n \odot x) \le s((n \odot x)^-) = 0$ which is a
contradiction. Therefore, $s(x) = 0$. Hence $\mbox{Ker}(s_I)
\subseteq \mbox{Ker}(s)$. Since $s_I$ and $s$ are extremal, their
kernels are maximal ideals, so that, $\mbox{Ker}(s_I) =
\mbox{Ker}(s),$ consequently, $s= s_I$.  Hence, $M$ admits only one
extremal state, $M = \bigcup_{t \in [0,1]_\mathbb H} M_t,$ and $M_t = I_t,$ where $I_t =s^{-1}(\{t\}),$ for $t \in [0,1]_\mathbb H,$  as stated.

\vspace{2mm}\noindent {\it Claim 3.} {\it If $a \in I_v$ and $b \in
I_t$ for $v<t, $ $v,t \in [0,1]_\mathbb H,$ then $a < b.$} \vspace{2mm}

Let $\hat s_I$ denote the (unique) extension of $s$ onto the
$\ell$-group $(G,u)$, that is, $s_I$ is a real-valued additive (in
our case preserving multiplication) mapping on $(G,u)$ preserving
the order on $G$, and $s_I(u) = 1.$

There are two cases: (1') $Cg = Cg(a\wedge b)$ and (2') $Cg \ne
Cg(a\wedge b)$.

(1') If $Cg = Cg(a\wedge b)$, then $a\wedge b \in g^{-1}Cg$, and
while $g^{-1}Cg$ is prime, $a \in g^{-1}Cg$ or $b\in g^{-1}Cg$. Then
$a = g^{-1}cg$ that gives $v = s_I(a) = \hat s_I(g^{-1}) + \hat
s_I(c) +\hat s_I(g)= 0$ which is a contradiction.  Similarly, $b \in
g^{-1}Cg$ gives the same contradiction.  Therefore (2') holds only.

(2') Transitivity guarantees the existence of an $h \in G$ such that
$Cgh = Cg(a\wedge b).$ Hence, $Cgh = Cghh^{-1}(a\wedge b)$ which
yields $ h^{-1}(a\wedge b) \in (gh)^{-1}Cgh$.  Since $h = g^{-1}cg
(a\wedge b)$ we have $\hat s_I(h) = \hat s_I(g^{-1}) + \hat s_I(c) +
\hat s_I(g) + \hat s_I(a\wedge b) = \hat s_I(a\wedge b) = s(a).$
Therefore,  $h^{-1}(a\wedge b) = (h^{-1} a) \wedge (h^{-1}b) \in
(gh)^{-1}C(gh).$ Since $(gh)^{-1}C(gh)$ is prime, and $h \in \hat
I$, we get $h^{-1}a \in (gh)^{-1}C(gh)$. Then $h^{-1}a =
(gh)^{-1}cgh$ for some $c \in C$, and $ga = ghh^{-1} a = cgh$, i.e.,
$Cga = Cgh$. But $Cga = Cgh = Cg(a\wedge b) \le Cgb$.

Combining (1')--(2'), we have $Cga \le Cgb$  for any $g \in G$,
consequently, $a\le b,$ which yields $a<b.$

Finally, using Claim 1 and Claim 3, we have $I_0 \leqslant I_v \leqslant  I_t \leqslant  I_1,$ for $v<t,$ $v,t \in [0,1]_\mathbb H\setminus \{0,1\},$ which proves $M= (M_t: t \in [0,1]_\mathbb H)$ and $M$  is $\mathbb H$-perfect. By (iv) of Theorem \ref{th:3.2}, we have that $M$ has a unique state.

(II)  The general case $M/I \cong \Gamma(\mathbb K,1)$, where $\mathbb K$ is a subgroup of $\mathbb H,$
follows the same ideas as that for $\mathbb K=\mathbb H$ proving $M$ is $\mathbb K$-perfect.
\end{proof}

\section{Strong $\mathbb H$-perfect PMV-algebras and Their Representation}

In this section, we introduce a stronger notion of $\mathbb H$-perfect PMV-algebras, called strong $\mathbb H$-perfect PMV-algebras, and we show when it can be represented in the form $\Gamma(\mathbb H\lex G,(1,0))$ for some unital $\ell$-group $G.$

We say that a PMV-algebra $M$ enjoys
{\it unique  extraction of roots of $1$} if $a,b \in M$ and $n a, nb$ exist in $M$, and $n a=1= n b$, then $a= b.$  Then every PMV-algebra $\Gamma(\mathbb H \lex G,(1,0)) $ enjoys unique  extraction of roots of $1$ for any $n\ge 1$
and for any $\ell$-group $G$.  Indeed, let $k(s,g) = (1,0)=k(t,h)$ for some $s,t \in [0,1]_\mathbb H,$ $g,h \in G$, $k\ge 1.$ Then $ks=1=kt$ which yields $s=t >0$, and $kg=0=kh$ implies $g=0=h.$

The following notion of a cyclic element was defined for PMV-algebras in \cite{Dv08, 225} and for pseudo effect algebras in \cite{DXY}.

Let $n\ge 1$ be an integer. An element $a$ of a PMV-algebra $M$ is said to be {\it cyclic of order} $n$ or simply {\it cyclic} if $na$ exists in $M$ and $na =1.$ If $a$ is a cyclic element of order $n$, then $a^- = a^\sim$, indeed, $a^- = (n-1) a = a^\sim$. It is clear that $1$ is a cyclic element of order $1.$

Let $M=\Gamma(G,u)$ for some unital $\ell$-group $(G,u).$ An element
$c\in M$ such that (a) $nc=u$ for some integer $n \ge 1,$ and (b) $c\in   C(H),$ where $C(H)$ is a commutative center of $H,$
is said to be a {\it strong cyclic element of order $n$.}

For example, the PMV-algebra $M:=\Gamma(\mathbb Q \lex G, (1,0)),$ for every integer $n\ge 1,$ $M$ has a unique cyclic element of order $n,$ namely $a_n =(\frac{1}{n},0).$ The PMV-algebra $\Gamma(\frac{1}{n}\mathbb Z, (1,0))$ for a prime number $n\ge 1,$ has  the only cyclic element of order $n,$ namely $(\frac{1}{n},0).$ If $M=\Gamma(G,u)$ and $G$ is a representable $\ell$-group, $G$ enjoys unique extraction of roots of $1,$ therefore, $M$ has at most one cyclic element of order $n.$ In general, a PMV-algebra $M$ can have two different cyclic elements of the same order. But if $M$ has a strong cyclic element of order $n,$ then it has a unique strong cyclic element of order $n$ and a unique cyclic element of order $n,$ \cite[Lem 5.2]{DvKo}.

The following notions were introduced in \cite{DvKo} for pseudo effect algebras.

We say that an $\mathbb H$-decomposition $(M_t: t\in [0,1]_\mathbb H)$ of $M$ has the {\it cyclic property} if there is a system of elements $(c_t\in M: t \in [0,1]_\mathbb H)$ such that (i) $c_t \in M_t$ for any $t \in [0,1]_\mathbb H,$ (ii) if $v+t \le 1,$ $v,t \in [0,1]_\mathbb H,$ then $c_v+c_t=c_{v+t},$ and (iii) $c_1=1.$ Properties: (a) $c_0=0;$ indeed, by (ii) we have $c_0+c_0=c_0,$ so that $c_0=0.$ (b) If $t=1/n,$ then $c_\frac{1}{n}$ is a cyclic element of order $n.$

Let $M =\Gamma(G,u),$ where $(G,u)$ is a unital $\ell$-group. An $\mathbb H$-decomposition $(M_t: t\in [0,1]_\mathbb H)$ of $M$ has the {\it strong cyclic property} if there is a system of elements $(c_t\in M: t \in [0,1]_\mathbb H)$ such that (i) $c_t \in M_t\cap C(G)$ for any $t \in [0,1]_\mathbb H,$ (ii) if $v+t \le 1,$ $v,t \in [0,1]_\mathbb H,$ then $c_v+c_t=c_{v+t},$ and (iii) $c_1=1.$ We recall that if $t=1/n,$ $c_\frac{1}{n}$ is a strong cyclic element of order $n.$

For example, let $M=\Gamma(\mathbb H \lex G,(1,0)),$ where $G$ is an $\ell$-group, and $M_t =\{(t,g): (t,g)\in M\}$ for $t \in [0,1]_\mathbb H.$ If we set $c_t =(t,0),$ $t \in [0,1]_\mathbb H,$ then the system $(c_t: t \in [0,1]_\mathbb H)$ satisfies (i)---(iii) of the strong cyclic property, and $(M_t: t \in [0,1]_\mathbb H)$ is an $\mathbb H$-decomposition of $M$ with the strong cyclic property.

Finally, we say that a PMV-algebra $M$ has the $\mathbb H$-{\it strong cyclic property} if there is an  $\mathbb H$-decomposition $(M_t: t \in [0,1]_\mathbb H)$ of $M$ with the strong cyclic property.

If $\mathbb H = \mathbb Q,$ we can show an equivalent definition for the $\mathbb Q$-strong cyclic property, see also \cite[Prop 7.1]{DvKo}. Namely, we say  that a PMV-algebra $M=\Gamma(G,u),$ where $(G,u)$ is a unital $\ell$-group,  enjoys the {\it strong}  $1$-{\it divisibility property} if, given integer $n \ge 1$, there is an element $a_n \in C(G)\cap M$ such that $na_n=1.$ We see that $a_n$ is a strong cyclic element of order $n$ which is unique, and we denote it by $a_n=\frac{1}{n}1.$ For any integer $m,$ $0\le m \le n,$ we write $m\frac{1}{n}1=:\frac{m}{n}1.$

\begin{proposition}\label{pr:3.3}
{\rm (1)} A PMV-algebra $M =\Gamma(G,u),$ where $(G,u)$ is a unital $\ell$-group, has the $\mathbb Q$-strong cyclic property if and only if $M$ has the strong $1$-divisibility property.

{\rm(2)} A PMV-algebra $M =\Gamma(G,u),$ where $(G,u)$ is a unital $\ell$-group, has the $\frac{1}{n}\mathbb Z$-strong cyclic property if and only if $M$ has a strong cyclic element of order $n.$
\end{proposition}

\begin{proof}
(1) It follows from \cite[Prop 7.1]{DvKo}.

(2) It follows from the definition of a strong cyclic element.
\end{proof}

Now we introduce a stronger notion of $\mathbb H$-perfect PMV-algebras which is inspired by an analogous one for PEAs' see \cite{DvKo}.
We say that a PMV-algebra $M$ is {\it strong} $\mathbb H$-{\it perfect} if $M$ possesses an $\mathbb H$-decomposition of $M$ having the strong cyclic property.

A prototypical example of a strong $\mathbb H$-perfect PMV-algebra is the following.

\begin{proposition}\label{pr:3.4}
Let $G$ be an  $\ell$-group. Then  the PMV-algebra
$$
\mathcal M_\mathbb H(G):=\Gamma(\mathbb H \lex G,(1,0)) \eqno(4.1)
$$
is a strong $\mathbb H$-perfect PMV-algebra.
\end{proposition}

We present a representation theorem for strong $\mathbb H$-perfect PMV-algebras by (4.1).

\begin{theorem}\label{th:3.5}
Let $M$ be a strong $\mathbb H$-perfect PMV-algebra.  Then there is a unique (up to isomorphism) $\ell$-group $G$  such that $M \cong \Gamma(\mathbb H \lex G,(1,0)).$
\end{theorem}

\begin{proof}
Since $M$ is a PMV-algebra, due to \cite[Thm 3.9]{151}, there is a unique unital (up to isomorphism of unital $\ell$-groups)  $\ell$-group $(H,u)$ such that $M = \Gamma(H,u).$ Assume $(M_t: t \in [0,1]_\mathbb H)$ is an $\mathbb H$-decomposition of $M$ with the strong cyclic property and with a given system of elements $(c_t \in M: t \in [0,1]_\mathbb H)$; due to Theorem \ref{th:3.2}, $(M_t: t \in [0,1]_\mathbb H)$ is unique.

By (v)--(vi) of Theorem \ref{th:3.2}, $M_0=\Infinit(M)$ is an associative cancellative semigroup satisfying conditions of Birkhoff \cite[Thm XIV.2.1]{Bir}, \cite[Thm II.4]{Fuc}
which guarantees that $M_0$ is a positive cone of
a unique (up to isomorphism) directed po-group $G$. Since $M_0$ is a lattice, we have that $G$ is an $\ell$-group.

By Theorem \ref{th:3.2}(iv), there is a unique $\mathbb H$-valued state $s.$ This state is extremal, therefore, by \cite[Prop 4.7]{156}, $s(a\wedge b)=\min\{s(a),s(b)\}$ for all $a,b \in M,$ and the same is true for its extension $\hat s$ onto $(H,u)$ and all $a,b \in H.$

Take the $\mathbb H$-strong cyclic PMV-algebra $\mathcal M_\mathbb H(G)$ defined by (4.1), and define a mapping $\phi: M \to \mathcal M_\mathbb H(G)$ by

$$
\phi(x):= (t, x - c_t)\eqno (4.2)
$$
whenever $x \in M_t$ for some $t \in [0,1]_\mathbb H,$ where $ x-c_t$ denotes the difference taken in the group $H.$

\vspace{2mm}
{\it Claim 1:} {\it  $\phi$ is a well-defined mapping.}
 \vspace{2mm}

Indeed, $M_0$  is in fact the positive cone of an $\ell$-group $G$ which is a subgroup of $H.$    Let $x \in M_t.$ For the element $x - c_t \in H,$ we define $(x-c_t)^+:= (x-c_t)\vee 0 = (x \vee c_t)-c_t \in M_0$ while $s((x \vee c_t)-c_t)=s(x \vee c_t)-s(c_t)= t-t=0$ and similarly $(x -c_t)^- := -((x-c_t)\wedge 0) = c_t - (x\wedge c_t) \in M_0.$ This implies that $x-c_t= (x-c_t)^+ - (x-c_t)^-\in G.$

\vspace{2mm}
{\it Claim 2:} {\it The mapping $\phi$ is an injective and surjective homomorphism of pseudo effect algebras.}

\vspace{2mm}

We have $\phi(0)=(0,0)$ and $\phi(1)=(1,0).$ Let $x \in M_t.$ Then $x^- \in M_{1-t},$ and $\phi(x^-) =(1-t, x -  c_{1-t}) = (1,0)-(t,x - c_t)=\phi(x)^-.$ In an analogous way, $\phi(x^\sim)=\phi(x)^\sim.$

Now let $x,y \in M$ and let $x+y$ be defined in $M.$ Then $x\in M_{t_1}$ and $y \in M_{t_2}.$ Since $x\le y^-,$ we have $t_1 \le 1-t_2$ so that $\phi(x) \le \phi(y^-)=\phi(y)^-$ which means  $\phi(x)+\phi(y)$ is defined in $\mathcal M_\mathbb H (G).$ Then $\phi(x+y) = (t_1+t_2, x+y - c_{t_1+t_2}) =
(t_1+t_2, x+y -(c_{t_1} + c_{t_2}))= (t_1,x-c_{t_1}) + (t_2,y- c_{t_2})=\phi(x)+\phi(y).$

Assume $\phi(x)\le \phi(y)$ for some $x\in M_{t}$ and $y \in M_v.$ Then $(t,x-c_t)\le (v, y - c_v).$ If $t=v,$ then $x-c_t\le y-c_t$ so that $x\le y.$  If $i<j,$ then $x \in M_t$ and $y\in M_v$ so that $x<y.$  Therefore, $\phi$ is injective.

To prove that $\phi$ is surjective, assume two cases: (i) Take $g \in G^+=M_0.$  Then $\phi(g)=(0,g).$ In addition $g^- \in M_1$ so that $\phi(g^-) = \phi(g)^-= (0,g)^- = (1,0)-(0,g)=(1,-g).$ (ii) Let $g \in G$ and $t$ with $0<t<1$ be given. Then $g = g_1-g_2,$ where $g_1,g_2 \in G^+=M_0.$ Since $c_t \in M_t,$ $g_1 + c_t$ exists in $M$ and it belongs to $M_t,$ and $g_2 \le g_1+c_t$ which yields $(g_1+c_t)- g_2 = (g_1+c_t)\minusli g_2 \in M_t.$  Hence, $g+c_t = (g_1 + c_t)\minusli g_2 \in M_t$ which entails $\phi(g+c_t)=(t,g).$

\vspace{2mm}
{\it Claim 3:}  {\it If $x\le y,$ then $\phi(y \minusli x)=\phi(y)\minusli \phi(x)$ and $\phi(x\minusre y)=\phi(x)\minusre \phi(y).   $}
\vspace{2mm}

It follows from the fact that $\phi$ is a homomorphism of PEAs.

\vspace{2mm}
{\it Claim 4:} $\phi(x\wedge y) = \phi(x)\wedge \phi(y)$ and $\phi(x\vee y)=\phi(x)\vee\phi(y).$
\vspace{2mm}

We have, $\phi(x), \phi(y) \ge \phi(x\wedge y).$ If $\phi(x), \phi(y) \ge \phi(w)$ for some $w \in M,$ we have $x,y \ge w$ and $x\wedge y \ge w.$ In the same way we deal with $\vee.$

\vspace{2mm}
{\it Claim 5:} {\it $\phi$ is a homomorphism of PMV-algebras.}
\vspace{2mm}

It is necessary to show that $\phi(x\oplus y)=\phi(x)\oplus \phi(y).$
It follows from the above claims and equality (2.2).

Consequently, $M$ is isomorphic to $\mathcal M_\mathbb H(G)$ as PMV-algebras.

If $M \cong \Gamma(\mathbb H\lex G',(1,0)),$ then $G$ and $G'$ are isomorphic $\ell$-groups in view of the categorical equivalence, see \cite[Thm 6.4]{151} or Theorem \ref{th:2.1}.
\end{proof}

\section{Categorical Equivalence of Strong $\mathbb H$-perfect PMV-algebras}\label{sec:5}

The categorical equivalence of strong $n$-perfect PMV-algebras with the category of $\ell$-group was established in \cite[Thm 7.7]{Dv08}. In this section, we generalize this result for the category of strong $\mathbb H$-perfect PMV-algebras. Our methods are similar to those used in \cite{Dv08}.

Let $\mathcal {SPPMV}_\mathbb H$ be the category of strong $\mathbb H$-perfect pseudo MV-algebras whose objects are strong $\mathbb H$-perfect pseudo MV-algebras and morphisms are homomorphisms of PMV-algebras. Now let $\mathcal G$ be the category whose objects are $\ell$-groups  and morphisms are homomorphisms of unital $\ell$-groups.

Define a mapping $\mathcal M_\mathbb H: \mathcal G \to  \mathcal {SPPMV}_\mathbb H$ as follows: for $G\in \mathcal G,$ let
$$
\mathcal M_\mathbb H(G):= \Gamma(\mathbb H\lex G,(1,0))
$$
and if $h: G \to G_1$ is an $\ell$-group homomorphism, then

$$
\mathcal M_\mathbb H(h)(t,g)= (t, h(g)), \quad (t,g) \in \Gamma(\mathbb H\lex G,(1,0)).
$$
It is easy to see that $\mathcal M_\mathbb H$ is a functor.

\begin{proposition}\label{pr:4.1}
$\mathcal M_\mathbb H$ is a faithful and full
functor from the category ${\mathcal G}$ of $\ell$-groups  into the
category $\mathcal{SPPMV}_\mathbb H$ of strong $\mathbb H$-perfect PMV-algebras.
\end{proposition}

\begin{proof}
Let $h_1$ and $h_2$ be two morphisms from $G$
into $G'$ such that $\mathcal M_\mathbb H(h_1) = \mathcal M_\mathbb H(h_2)$. Then
$(0,h_1(g)) = (0,h_2(g))$ for any $g \in G^+$, consequently $h_1 =
h_2.$

To prove that $\mathcal M_\mathbb H$ is a full  functor, suppose that $f$ is a morphism from a strong $\mathbb H$-perfect PMV-algebra
$\Gamma(\mathbb H\lex G, (1,0))$ into another one $\Gamma(\mathbb
H\lex G_1, (1,0)).$  Then $f(0,g)
= (0,g')$ for a unique $g' \in G'^+$. Define a mapping $h:\ G^+ \to
G'^+$ by $h(g) = g'$ iff $f(0,g) =(0,g').$ Then $h(g_1+g_2) = h(g_1)
+ h(g_2)$ if $g_1,g_2 \in G^+.$
Assume now that $g \in G$ is arbitrary. Then $g = g_1 -
g_2 = g_1'-g_2'$, where $g_1, g_2, g_1', g_2' \in G^+,$ which gives
$g_1 +g_2' = g_1' + g_2$, i.e., $h(g) = h(g_1) - h(g_2)$ is a
well-defined extension of $h$ from $G^+$ onto $G$.

Let $0\le g_1 \le g_2.$ Then $(0,g_1)\le (0,g_2),$
which means  $h$ is a mapping preserving the partial order.

We have yet to show that $h$ preserves $\wedge$ in $G$, i.e., $h(a \wedge b) = h(a) \wedge h(b)$ whenever $a,b \in G.$ Let $a=  a^+- a^-$ and $b=
b^+- b^-$, and $a =-a^- +a^+$, $b = -b^- + b^+$. Since , $h((a^+
+b^-) \wedge (a^- + b^+)) = h(a^+ +b^-) \wedge h(a^- + b^+).$
Subtracting $h(b^-)$ from the right hand and $h(a^-)$ from the left
hand, we obtain the statement  in question.

Finally, we have established that $h$ is a homomorphism of $\ell$-groups, and $\mathcal M_\mathbb H(h) = f$ as claimed.
\end{proof}

 We recall that by a {\it universal group}  for a
PMV-algebra $M$ we mean a pair $(G,\gamma)$ consisting of an
$\ell$-group $G$ and a $G$-valued measure $\gamma :\, M\to G^+$
(i.e., $\gamma (a+b) = \gamma(a) + \gamma(b)$ whenever $a+b$ is
defined in $M$) such that the following conditions hold: {\rm (i)}
$\gamma(M)$ generates ${ G}$. {\rm (ii)} If $H$ is a group and
$\phi:\, M\to H$ is an $H$-valued measure, then there is a group
homomorphism ${\phi}^*:{ G}\to H$ such that $\phi ={\phi}^*\circ
\gamma$.

Due to \cite{151}, every PMV-algebra admits a universal group,
which is unique up to isomorphism, and $\phi^*$ is unique. The
universal group for $M = \Gamma(G,u)$ is $(G,id)$ where $id$ is the
embedding of $M$ into $G$.

Let $\mathcal A$ and $\mathcal B$ be two categories and let $f:\mathcal A \to \mathcal B$ be a morphism. Suppose that $g,h$ be two morphisms from $\mathcal B$ to $\mathcal A$ such that $g\circ f = id_\mathcal A$ and $f\circ h = id_\mathcal B,$ then $g$ is a {\it left-adjoint} of $f$ and $h$ is a {\it right-adjoint} of $f.$

\begin{proposition}\label{pr:4.2}
The functor $\mathcal  M_\mathbb H$ from the
category ${\mathcal  G}$ into the category $\mathcal{SPPMV}_\mathbb H$ has  a left-adjoint.
\end{proposition}

\begin{proof}
We show, for a strong $\mathbb H$-perfect  PMV-algebra $M$ with an  $\mathbb H$-decomposition $(M_t: t \in [0,1]_\mathbb H)$ and a system $(c_t: t \in [0,1]_\mathbb H)$  of elements of $M$ satisfying (i)--(iii) of the strong cyclic property, there is a universal arrow $(G,f)$, i.e., $G$ is an object in $\mathcal G$ and $f$ is a homomorphism from the PMV-algebra
$M$ into ${\mathcal  M}_\mathbb H(G)$ such that if $G'$ is an object from ${\mathcal G}$ and $f'$ is a homomorphism from $M$ into ${\mathcal  M}_\mathbb H(G')$, then
there exists a unique morphism $f^*:\, G \to G'$ such that ${\mathcal
M}_\mathbb H(f^*)\circ f = f'$.

By Theorem \ref{th:3.5}, there is a unique (up to isomorphism of $\ell$-groups) $\ell$-group $G$  such that $M \cong \Gamma(\mathbb H \lex G,(1,0)).$ By \cite[Thm 5.3]{151}, $(\mathbb H \lex G, \gamma)$ is a universal group for $M,$ where $\gamma: M \to  \Gamma(\mathbb H \lex G, (1,0))$ is defined by $\gamma(a) = (t,a -c_t),$ if $a \in M_t.$
\end{proof}

Define a mapping ${\mathcal  P}_\mathbb H: \mathcal  {SPPMV}_\mathbb H \to {\mathcal  G}$
via ${\mathcal  P}_\mathbb H(M) := G$ whenever $(\mathbb H\lex  G, f)$ is a
universal group for $M$. It is clear that if $f_0$ is a morphism
from the PMV-algebra $M$ into another one $N$, then $f_0$ can be uniquely extended to an $\ell$-group homomorphism ${\mathcal  P}_\mathbb H(f_0)$ from $G$ into $G_1$, where $(\mathbb H
\lex G_1, f_1)$ is a universal group for the strong $\mathbb H$-perfect
PMV-algebra $N$.

\begin{proposition}\label{pr:4.3}
The mapping ${\mathcal  P}_\mathbb H$ is a functor from the
category $\mathcal {SPPMV}_\mathbb H$ into the category ${\mathcal  G}$ which is a
left-adjoint of the functor ${\mathcal  M}_\mathbb H.$
\end{proposition}

\begin{proof}
 It follows from the properties of the
universal group.
\end{proof}

Now we present  the main result on a categorical equivalence of the
category of strong $\mathbb H$-perfect PMV-algebras and the category of $\mathcal G.$

\begin{theorem}\label{th:4.4}
The functor ${\mathcal  M}_\mathbb H$ defines a categorical
equivalence of the category ${\mathcal  G}$   and the
category $\mathcal {SPPMV}_\mathbb H$ of strong $\mathbb H$-perfect PMV-algebras.

In addition, suppose that $h:\ {\mathcal  M}_\mathbb H\mathbb (G) \to {\mathcal  M}_\mathbb H(H)$ is a
homomorphism of pseudo effect algebras, then there is a unique homomorphism
$f:\ G \to H$ of unital po-groups such that $h = {\mathcal  M}_\mathbb H(f)$, and
\begin{enumerate}
\item[{\rm (i)}] if $h$ is surjective, so is $f$;
 \item[{\rm (ii)}] if $h$ is  injective, so is $f$.
\end{enumerate}
\end{theorem}

\begin{proof}
According to \cite[Thm IV.4.1]{MaL}, it is
necessary to show that, for a strong $\mathbb H$-perfect PMV-algebra $M$, there is an
object $G$ in ${\mathcal  G}$ such that ${\mathcal  M}_\mathbb H(G)$ is isomorphic to $M$. To show that, we take a universal group $(\mathbb H
\lex G, f)$. Then ${\mathcal  M}_\mathbb H(G)$ and $M$ are isomorphic.
\end{proof}

Theorem \ref{th:4.4} entails directly the following statement.

\begin{corollary}\label{co:4.5}
If $\mathbb H_1$ and $\mathbb H_2$ are two subgroups of $\mathbb R$ containing the number 1, then the categories $\mathcal{SPPMV}_{\mathbb H_1},$ $\mathcal{SPPMV}_{\mathbb H_2}$ and the category $\mathcal G$ of $\ell$-groups are mutually categorically equivalent.
\end{corollary}

\begin{theorem}\label{th:4.6} Let $G$ be a  doubly transitive
$\ell$-group.  Then  ${\mathcal V}(\mathcal{SPPMV}_\mathbb H) = {\mathcal V}({\mathcal M}_\mathbb H(G)) .$

In particular, an identity holds in every strong $\mathbb H$-perfect
PMV-algebra if and only if it holds in ${\mathcal M}_\mathbb H(G).$

\end{theorem}

\begin{proof}   Let $G$ be a doubly transitive
$\ell$-group, and define a strong $\mathbb H$-perfect PMV ${\mathcal M}_\mathbb H (G)$ by (4.1).

Let $M$ be a strong $\mathbb H$-perfect PMV-algebra.  Due to Theorem
\ref{th:3.5}, there is a unique (up to isomorphism of unital $\ell$-groups) $\ell$-group $G_M$ such that $M=
{\mathcal M}_\mathbb H(G_M).$ Since every doubly transitive $\ell$-group
generates the variety $\mathcal G$ of $\ell$-groups, \cite[Lem.
10.3.1]{Gla}, there exist a homomorphism $f$ of $\ell$-groups  and
an $\ell$-group $K$ such that $f(K) = G_M$ and $K \subseteq G^J$,
where $J$ is an index set. Due to Theorem \ref{th:4.4},
$M= {\mathcal M}_\mathbb H (G_M)={\mathcal M}_\mathbb H (f)({\mathcal M}_\mathbb H (K)).$

Define a map $\rho:\ {\mathcal  M}_\mathbb H (G^J)\to ({\mathcal  M}_\mathbb H (G))^J$ via
$\rho(0, (g_j)_{j\in J}) = \{(0,g_j)\}_{j\in J}$ and
$\rho(1,(-g_j)_{j\in J}) = \{(1,-g_j)\}_{j\in J}$ for $g_j \in G^+$,
and $\rho(t,g_j) =\{(t,g_j)\}_{j\in J}$, $t \in [0,1]_\mathbb H \setminus \{0,1\},$ $g_j \in G$ for $j
\in J.$ Then $\rho$ is an embedding, and ${\mathcal  M}_\mathbb H (G^J) \in {\mathcal
V}({\mathcal  M}_\mathbb H (G)).$ Since ${\mathcal  M}_\mathbb H (K)$ is a subalgebra of ${\mathcal M}_\mathbb H(G^J)$, we have ${\mathcal  M}_\mathbb H (K) \in {\mathcal  V}({\mathcal  M}_\mathbb H (G))$ and
$M \in {\mathcal  V}({\mathcal  M}_\mathbb H (G))$ because it is a homomorphic image of
${\mathcal  M}_\mathbb H (K)\in {\mathcal  V}({\mathcal  M}_\mathbb H (G)).$
\end{proof}

An example of a doubly transitive  permutation $\ell$-group is the
system of all automorphisms, $\mbox{Aut}(\mathbb R)$, of the real
line $\mathbb R$, or the next example:

Let $u\in \mbox{Aut}(\mathbb R)$ be the translation $t u=t+1$, $t
\in \mathbb R$, and
$${\rm BAut}(\mathbb R)=\{g\in\mbox{Aut}(\mathbb R):\ \exists\ n\in{\mathbb
N},\,\,\, u^{-n}\le g\le u^n\}.
$$
Then $(\mbox{\rm BAut}(\mathbb R),u)$ is a   doubly transitive
unital $\ell$-permutation group, and  according to \cite[Cor.
4.9]{DvHo}, the variety of PMV-algebras generated by
$\Gamma(\mbox{\rm BAut}(\mathbb R),u)$ is the variety of all PMV-algebras.

\section{Weak $\mathbb H$-perfect PMV-algebras} 

In this section, we introduce another family of $\mathbb H$-perfect PMV-algebras, called weak $\mathbb H$-perfect PMV-algebras. They can be represented in the form $\Gamma(\mathbb H \lex G,(1,b)),$ where $G$ is an $\ell$-group and $0<b  \in G^+.$ Such PVM-algebras were studied in \cite{Dv08} for the case when $\mathbb H$ is a cyclic subgroup of $\mathbb R.$

We say that an $\mathbb H$-perfect pseudo MV-algebra $M=(M_t: t \in [0,1]_\mathbb H),$ where $M=\Gamma(G,u),$ is {\it weak} if there is a system $(c_t: t \in [0,1]_\mathbb H)$ of elements of $M$ such that (i) $c_0=0,$ (ii) $c_t \in C(G)\cap M_t,$ for any $t \in [0,1]_\mathbb H,$ and (iii) $c_{v+t}=c_v+c_t$ whenever $v+t \le 1.$  We note that in contrast to strong cyclic property, we do not assume $c_1=1.$ In addition, a weak $\mathbb H$-perfect PMV-algebra $M$ is strong iff $c_1 =1.$

Whereas every strong $\mathbb H$-perfect PMV-algebra is symmetric, for weak $\mathbb H$-perfect PMV-algebras this is not necessarily a case.

For example, if $g_0$ is a positive element of an $\ell$-group $G$ such that $g_0 \not\in C(G),$ then $M= \Gamma(\mathbb H\lex G,(1,g_0))$ is a weak $\mathbb H$-perfect PMV-algebra which is neither symmetric, nor strong; we set $c_t =(t,0)$ for any $t \in [0,1]_\mathbb H.$ Then $c_1 =(1,0) < (1,g_0).$

\begin{theorem}\label{th:6.1}
Let $M=(M_t: t \in [0,1]_\mathbb H)$ be a
weak  $\mathbb H$-perfect PMV-algebra which is not strong. Then there is a
unique (up to isomorphism) $\ell$-group $G$ with an element $b\in
G^+$, $b
>0,$ such that $M \cong \Gamma(\mathbb Z \lex   G,(n,b)).$
\end{theorem}

\begin{proof}
Assume $M= \Gamma(H,u)$ for some unital
$\ell$-group $(H,u).$ As in the proof of Theorem \ref{th:3.5}, we
can found a unique (up to isomorphism) $\ell$-group $G$ such that
$\mbox{Infinit}(M) = M_0$ is the positive cone of $G,$ moreover, $G$
is an $\ell$-subgroup of $H.$ We recall that if  $s$ is a unique
state on $M$, it can be extended to a unique state, $\hat s$, on the
unital $\ell$-group $(G,u)$. Since by (iv) Theorem \ref{th:3.2}, $M_0= \mbox{Ker}(s)$, we have $G=\mbox{Ker}(\hat s).$

Since $M$ is not strong,  then $c_1 <1=:u.$ Set $b =  u \minusli c_1 = 1-c_1 \in M_0\setminus\{0\},$ and define a mapping $h: \
M\to \Gamma(\mathbb Z \lex   G,(1,b))$ as follows
$$
\phi(x) =(t,x-c_t)\eqno(6.1)
$$
whenever $x \in M_t;$ we note that the subtraction $x-c_t$ is defined in the $\ell$-group $H.$  In the same way as in (3.2), we can show that $\phi$ is a well-defined mapping.

We have (1) $\phi(0) =(0,0)$, (2) $\phi(1) = (1,1 -c_1)
= (1,b)$, (3) $\phi(c_t) = (t,0),$ (4)  $\phi(x^\sim) =(1-t, -x +u -c_{1-t})
= (1-t, -x +b +c_t),$  $\phi(x)^\sim =  -\phi(x) +(1,b) = -(t, x-c_t) +(1,b)
= (1-t, -x+b +c_t)$ and similarly (5) $\phi(x^-) = \phi(x)^-.$

Using the same steps as those used in the proofs of all claims of the proof of Theorem \ref{th:3.5}, we can prove that $\phi$ is an injective and surjective homomorphism of pseudo MV-algebras as was claimed.
\end{proof}

We note that Theorem \ref{th:6.1} is a generalization of Theorem \ref{th:3.5}, because Theorem \ref{th:3.5} in fact follows from Theorem \ref{th:6.1} when we have $b=0.$

Finally, let  $\mathcal{WPPMV}_\mathbb H$ be the category of weak $\mathbb H$-perfect PMV-algebras whose objects are weak $\mathbb H$-perfect PMV-algebras and morphisms are homomorphisms of PMV-algebras. Similarly, let ${\mathcal L}_{\rm b}$ be the category whose objects are couples $(G,b),$ where $G$ is an $\ell$-group and $b$ is a fixed element from $G^+$, and morphisms are $\ell$-homomorphisms of $\ell$-groups preserving fixed elements $b$.

Define a mapping $\mathcal F_\mathbb H$ from the category $\mathcal L_{\rm b}$ into the category $\mathcal{WPPMV}_\mathbb H$ as follows:

Given $(G,b) \in {\mathcal L}_{\rm b},$ we set
$$
{\mathcal F}_\mathbb H(G,b) := \Gamma(\mathbb H\lex  G,(1,b)),\eqno(6.2)
$$
and if $h:(G,b)\to (G_1,b_1),$ then
$$
\mathcal F_\mathbb F(h)(t,g)=(t,h(g)),\quad (t,g) \in \Gamma(\mathbb H\lex G,(1,b)).
$$
It is easy to see that $\mathcal F_\mathbb H$ is a functor.

In the same way as the categorical equivalence of strong $\mathbb H$-perfect PMV-algebras was proved in Section \ref{sec:5}, we can prove the following theorem.

\begin{theorem}\label{th:6.2} The functor ${\mathcal F}_\mathbb H$ defines a categorical
equivalence of the category ${\mathcal L}_{\rm b}$ and the category
$\mathcal{WPPMV}_\mathbb H$ of weak $\mathbb H$-perfect PMV-algebras.
\end{theorem}

\end{document}